%% file: main.tex
\documentclass[ejs,preprint]{imsart}

\startlocaldefs
\usepackage{tommy}

\endlocaldefs

\begin{document}

\begin{frontmatter}
\title{Discrete minimax estimation with trees}
\runtitle{Discrete minimax estimation with trees}
\begin{aug}
  \author{\fnms{Luc} \snm{Devroye}\ead[label=e1]{lucdevroye@gmail.com}\thanksref{t1}}
  \and
  \author{\fnms{Tommy} \snm{Reddad}\ead[label=e2]{tommy.reddad@gmail.com}\thanksref{t2}}
  \affiliation{McGill University}

  \thankstext{t1}{Supported by NSERC Grant A3456.}
  \thankstext{t2}{Supported by NSERC PGS D scholarship 396164433.}
  
  \address{School of Computer Science \\
    McGill University \\
    3480 University Street \\
    Montr\'{e}al, Qu\'{e}bec, Canada \\
    H3A 2A7 \\
    \printead{e1,e2}
    }
  \runauthor{L. Devroye and T. Reddad}
\end{aug}

\begin{abstract}
  We propose a simple recursive data-based partitioning scheme which
  produces piecewise-constant or piecewise-linear density estimates on
  intervals, and show how this scheme can determine the optimal $L_1$
  minimax rate for some discrete nonparametric classes.
\end{abstract}

\begin{keyword}[class=MSC]
\kwd{60G07}
\end{keyword}

\begin{keyword}
\kwd{density estimation}
\kwd{minimax theory}
\kwd{discrete probability distribution}
\kwd{Vapnik-Chervonenkis dimension}
\kwd{monotone density}
\kwd{convex density}
\kwd{histogram}
\end{keyword}

\tableofcontents

\end{frontmatter}

\input{introduction}
\input{summary}
\input{nonincreasing}
\input{convex}
\input{discussion}

\appendix
\input{appendix}

\bibliographystyle{imsart-number}
\bibliography{main}

\end{document}

%% file: introduction.tex
\section{Introduction}

\emph{Density estimation} or \emph{distribution learning} refers to
the problem of estimating the unknown probability density function of
a common source of independent sample observations. In any interesting
case, we know that the unknown source density may come from a known
class. In the \emph{parametric} case, each density in this class can
be specified using a bounded number of real parameters, \eg the class
of all Gaussian densities with any mean and any variance. The
remaining cases are called \emph{nonparametric}. Examples of
nonparametric classes include bounded monotone densities on $[0, 1]$,
$L$-Lipschitz densities for a given constant $L > 0$, and log-concave
densities, to name a few. By \emph{minimax estimation}, we mean
density estimation in the minimax sense, \ie we are interested in the
existence of a density estimate which minimizes its approximation
error, even in the worst case.

There is a long line of work in the statistics literature about
density estimation, and a growing interest coming from the theoretical
computer science and machine learning communities; for a selection of
new and old books on this topic, see \cite{devroye-course,
  devroye-gyorfi, comb-methods, groeneboom-book, scott,
  silverman}. The study of nonparametric density estimation began as
early as in the 1950's, when Grenander~\cite{grenander} described and
studied properties of the maximum likelihood estimate of an unknown
density taken from the class of bounded monotone densities on
$[0, 1]$. Grenander's estimator and this class received much further
treatment over the years, in particular by Prakasa Rao~\cite{prakasa},
Groeneboom~\cite{groeneboom}, and Birg\'e~\cite{birge-order,
  birge-risk, birge}, who identified the optimal $L_1$-error minimax
rate up to a constant factor, and also gave an efficient adaptive
estimator which worked even when the boundedness parameter was
unknown. Since then, countless more nonparametric classes have been
studied, and many different all-purpose methods have been developed to
obtain minimax results about these classes: for the construction of
density estimates, see \eg the maximum likelihood estimate, skeleton
estimates, kernel estimates, and wavelet estimates, to name a few; and
for minimax rate lower bounds, see \eg the methods of Assouad, Fano,
and Le Cam~\cite{devroye-course, devroye-gyorfi, comb-methods,
  yu-survey}. See \cite{bellec, chat, gao, gunt} for recent related
works in nonparametric shape-constrained regression.

One very popular style of density estimate is the \emph{histogram}, in
which the support of the random data is partitioned into bins, where
each bin receives a weight proportional to the number of data points
contained within, and such that the estimate is constant with the
given weight along each bin. Then, the selection of the bins
themselves becomes critical in the construction of a good histogram
estimate. Birg\'e~\cite{birge-risk} showed how histograms with
carefully chosen exponentially increasing bin sizes will have
$L_1$-error within a constant factor of the optimal minimax rate for
the class of bounded non-increasing densities on $[0, 1]$. In general,
the right choice of an underlying partition for a histogram estimate
is not obvious.

In this work, we devise a recursive data-based approach for
determining the partition of the support for a histogram estimate of
discrete non-increasing densities. We also use a similar approach to
build a piecewise-linear estimator for discrete non-increasing convex
densities---see Anevski~\cite{anevski},
Jongbloed~\cite{jongbloed-thesis}, and Groeneboom, Jongbloed, and
Wellner~\cite{groeneboom-convex} for works concerning the maximum
likelihood and minimax estimation of continuous non-increasing convex
densities. Both of our estimators are \emph{minimax-optimal}, \ie
their minimax $L_1$-error is within a constant factor of the optimal
rate. Recursive data-based partitioning schemes have been extremely
popular in density estimation since the 1970's with
Gessaman~\cite{gessaman}, Chen and Zhao~\cite{chen}, Lugosi and
Nobel~\cite{lugosi-nobel}, and countless others, with great interest
coming from the machine learning and pattern recognition
communities~\cite{pattern}. Still, it seems that most of the
literature involving recursive data-based partitions are not
especially concerned with the rate of convergence of density
estimates, but rather other properties such as consistency under
different recursive schemes. Moreover, most of the density estimation
literature is concerned with the estimation of continuous probability
distributions. In discrete density estimation, not all of the
constructions or methods used to develop arguments for analogous
continuous classes will neatly apply, and in some cases, there are
discrete phenomena that call for a different approach. See Jankowski
and Wellner~\cite{jank} for a recent treatment on the properties of a
variety of estimators of discrete non-increasing densities.

%% file: summary.tex
\section{Preliminaries and summary}

Let $\sF$ be a given class of probability densities with respect to a
base measure $\mu$ on the measurable space $(\sX, \Sigma)$, and let
$f \in \sF$. If $X$ is a random variable taking values in
$(\sX, \Sigma)$, we write $X \sim f$ to mean that
\[
  \Pr\{X \in A\} = \int_A f \dif \mu , \qquad \text{for each
    $A \in \Sigma$.}
\]
The notation $X_1, \dots, X_n \iid f$ means that $X_i \sim f$ for each
$1 \le i \le n$, and that $X_1, \dots, X_n$ are mutually independent.

Typically in density estimation, either $\sX = \R^d$, $\Sigma$ is the
Borel $\sigma$-algebra, and $\mu$ is the Lebesgue measure, or $\sX$ is
countable, $\Sigma = \sP(\sX)$, and $\mu$ is the counting measure. The
former case is referred to as the \emph{continuous setting}, and the
latter case as the \emph{discrete setting}, where $f$ is more often
called a \emph{probability mass function} in the
literature. Throughout this paper, we will only be concerned with the
discrete setting, and even so, we still refer to $\sF$ as a class of
densities, and $f$ as a density itself. Plainly, in this case,
$X \sim f$ signifies that
\[
  \Pr\{X \in A\} = \sum_{x \in A} f(x) , \qquad \text{for each $A \in \sP(\sX)$.}
\]

Let $f \in \sF$ be unknown. Given the $n$ samples
$X_1, \dots, X_n \iid f$, our goal is to create a \emph{density
  estimate}
\[
  \hat{f}_n \colon \sX^n \to \R^\sX ,
\]
such that the probability measures corresponding to $f$ and
$\hat{f}_n(X_1, \dots, X_n)$ are close in \emph{total variation (TV)
  distance}, where for any probability measures $\mu, \nu$, their
TV-distance is defined as
\begin{equation}
  \TV(\mu, \nu) = \sup_{A \in \Sigma} |\mu(A) - \nu(A)| . \eqlabel{prob-interp}
\end{equation}
The TV-distance has several equivalent definitions; importantly, if
$\mu$ and $\nu$ are probability measures with corresponding densities
$f$ and $g$, then
\begin{align}
  \TV(\mu, \nu) &= \|f - g\|_1/2 , \eqlabel{l1-interp} \\
                &= \inf_{(X, Y) \colon X \sim f, Y \sim g} \Pr\{X \neq Y\} , \eqlabel{coupling-interp}
\end{align}
where for any function $h \colon \sX \to \R$, we define the $L_1$-norm
of $h$ as
\[
  \|h\|_1 = \sum_{x \in \sX} |h(x)| .
\]
(In the continuous case, this sum is simply replaced with an
integral.) In view of the relation between TV-distance and $L_1$-norm
in \eqref{l1-interp}, we will abuse notation and write
\[
  \TV(f, g) = \|f - g\|_1/2 .
\]

There are various possible measures of dissimilarity between
probability distributions which can be considered in density
estimation, \eg the Hellinger distance, Wasserstein distance,
$L_p$-distance, $\chi^2$-divergence, Kullback-Leibler divergence, or
any number of other divergences; see Sason and Verd\'{u}~\cite{verdu}
for a survey on many such functions and the relations between
them. Here, we focus on the TV-distance due to its several appealing
properties, such as being a metric, enjoying the natural probabilistic
interpretation of \eqref{prob-interp}, and having the coupling
characterization \eqref{coupling-interp}.

If $\hat{f}_n$ is a density estimate, we define the \emph{risk} of the
estimator $\hat{f}_n$ with respect to the class $\sF$ as
\[
  \sR_n(\hat{f}_n, \sF) = \sup_{f \in \sF} \E\bigl\{\TV(\hat{f}_n(X_1, \dots, X_n), f)\bigr\} ,
\]
where the expectation is over the $n$ i.i.d.\ samples from $f$, and
possible randomization of the estimator. From now on we will omit the
dependence of $\hat{f}_n$ on $X_1, \dots, X_n$ unless it is not
obvious. The \emph{minimax risk} or \emph{minimax rate} for $\sF$ is
the smallest risk over all possible density estimates,
\[
  \sR_n(\sF) = \inf_{\hat{f}_n} \sR_n(\hat{f}_n, \sF) .
\]

We can now state our results precisely. Let $k \in \N$ and let $\sF_k$
be the class of non-increasing densities on $\{1, \dots, k\}$, \ie set
of of all probability vectors $f \colon \{1, \dots, k\} \to \R$ for
which
\begin{equation}
  f(x + 1) \le f(x) , \qquad \text{for all $x \in \{1, \dots, k - 1\}$.} \eqlabel{decr}
\end{equation}
\begin{thm}\thmlabel{monotone-result}
  Let $f \colon \N \times \N \to \R$ be
  \[
    f(n, k) = \left\{
      \begin{array}{ll}
        \sqrt{k/n} & \mbox{if $2 \le k < 2 n^{1/3}$,} \\
        {\left( \frac{\log_2 (k/n^{1/3})}{n} \right)}^{1/3} & \mbox{if $2 n^{1/3} \le k < n^{1/3} 2^n$,} \\
        1 & \mbox{if $n^{1/3} 2^n \le k$.}
      \end{array} \right.
  \]
  There is a universal constant $C \ge 1$ such that, for sufficiently
  large $n$ not depending on $k$,
  \[
    \frac{1}{C} \le \frac{\sR_n(\sF_k)}{f(n, k)} \le C .
  \]
\end{thm}
Let $\sG_k$ be the class of all non-increasing convex densities on
$\{1, \dots, k\}$, so each $f \in \sG_k$ satisfies \eqref{decr} and
\[
  f(x) - 2 f(x + 1) + f(x + 2) \ge 0 , \qquad \text{for all $x \in \{1, \dots, k - 2\}$.} 
\]
\begin{thm}\thmlabel{convex-result}
  Let $g \colon \N \times \N \to \R$ be
  \[
    g(n, k) = \left\{
      \begin{array}{ll}
        \sqrt{k/n} & \mbox{if $2 \le k < 3 n^{1/5}$ ,} \\
        {\left(\frac{\log_3 (k/n^{1/5})}{n}\right)}^{2/5} & \mbox{if $3 n^{1/5} \le k < n^{1/5} 3^n$,} \\
        1 & \mbox{if $n^{1/5} 3^n \le k$.}
      \end{array} \right.
  \]
  There is a universal constant $C \ge 1$ such that, for sufficiently
  large $n$ not depending on $k$,
  \[
    \frac{1}{C} \le \frac{\sR_n(\sG_k)}{g(n, k)} \le C .
  \]
\end{thm}
We emphasize here that the above results give upper and lower bounds
on the minimax rates $\sR_n(\sF_k)$ and $\sR_n(\sG_k)$ which are
within universal constant factors of one another, for the entire range
of $k$.

Our upper bounds will crucially rely on the next results, which allow
us to relate the minimax rate of a class to an old and well-studied
combinatorial quantity called the \emph{Vapnik-Chervonenkis (VC)
  dimension}~\cite{vc}: For $\sA \subseteq \sP(\sX)$ a family of
subsets of $\sX$, the VC-dimension of $\sA$, denoted by $\VC(\sA)$, is
the size of the largest set $X \subseteq \sX$ such that for every
$Y \subseteq X$, there exists $B \in \sA$ such that $X \cap B =
Y$. See, \eg the book of Devroye and Lugosi~\cite{comb-methods} for
examples and applications of the VC-dimension in the study of density
estimation.

\begin{thm}[Devroye, Lugosi~\cite{comb-methods}]\thmlabel{minimum-distance-risk}
  Let $\sF$ be a class of densities supported on $\sX$, and let
  $\sF_\Theta = \{f_{n, \theta} \colon \theta \in \Theta\}$ be a class
  of density estimates satisfying
  $\sum_{x \in \sX} f_{n, \theta}(x) = 1$ for every
  $\theta \in \Theta$. Let $\sA_\Theta$ be the \emph{Yatracos class}
  of $\sF_\Theta$,
  \[
    \sA_\Theta = \Big\{ \{ x \in \sX \colon f_{n, \theta}(x) > f_{n, \theta'}(x)\} \colon \theta \neq \theta' \in \Theta \Big\} .
  \]
  For $f \in \sF$, let $\mu$ be the probability measure corresponding
  to $f$. Let also $\mu_n$ be the empirical measure based on
  $X_1, \dots, X_n \iid f$, where
  \[
    \mu_n(A) = \frac{1}{n} \sum_{i = 1}^n \mathbf{1}\{X_i \in A\} ,
    \qquad \text{for $A \in \sA_\Theta$}.
  \]
  Then, there is an estimate $\psi_n$ for which
  \[
    \TV(\psi_n, f) \le 3 \inf_{\theta \in \Theta} \TV(f_{n, \theta}, f) + 2 \sup_{A \in \sA_\Theta} |\mu_n(A) - \mu(A)| +
    \frac{3}{2 n} .
  \]
\end{thm}

The estimate $\psi_n$ in \thmref{minimum-distance-risk} is called the
\emph{minimum distance estimate} in \cite{comb-methods}---we omit the
details of its construction, though we emphasize that if computing
$\int_A f_{n, \theta}$ takes one unit of computation for any $\theta$
and $A$, then selecting $\psi_n$ takes time polynomial in the size of
$\sA$, which is often exponential in the quantities of interest; for
instance, if $\sA$ is the Yatracos class of $\sF_k$, then a simple
construction shows that $\sA$ contains all subsets of
$\{1, \dots, 2 \floor{k/2} \}$ containing only odd numbers, whence
\[
  2^{\floor{k/2}} \le |\sA| \le 2^k ,
\]
where the upper bound is trivial.

\begin{thm}[Devroye, Lugosi~\cite{comb-methods}]\thmlabel{tv-emp}
  Let $\sF, \sX, f, \mu, \mu_n$ be as in
  \thmref{minimum-distance-risk}, and let $\sA \subseteq
  \sP(\sX)$. Then, there is a universal constant $c > 0$ for which
  \[
    \sup_{f \in \sF} \E\left\{ \sup_{A \in \sA} |\mu_n(A) - \mu(A)| \right\} \le c \sqrt{\frac{\VC(\sA)}{n}} .
  \]
\end{thm}

\begin{rem}
  The quantity $\sup_{A \in \sA} |\mu(A) - \nu(A)|$ in \thmref{tv-emp}
  is precisely equal to $\TV(\mu, \nu)$ if $\sA$ is the Borel
  $\sigma$-algebra on $\sX$.
\end{rem}

\begin{corr}\corrlabel{minimum-distance-risk}
  Let $\sF, \sA_\Theta, f_{n, \theta}, f, \mu, \mu_n$ be as in
  \thmref{minimum-distance-risk}. Then, there is a universal constant
  $c > 0$ for which
  \[
    \sR_n(\sF) \le 3 \sup_{f \in \sF} \inf_{\theta \in \Theta} \TV(f_{n, \theta}, f) + c \sqrt{\frac{\VC(\sA_\Theta)}{n}} + \frac{3}{2n} .
  \]
\end{corr}

%% file: nonincreasing.tex
\section{Non-increasing densities}\seclabel{monotone}

This section is devoted to presenting a proof of the upper bound of
\thmref{monotone-result}. The lower bound is proved in
\appref{monotone-lower} using a careful but standard application of
Assouad's Lemma~\cite{assouad}. Part of our analysis in proving
\thmref{monotone-result} will involve the development of an explicit
efficient estimator for a density in $\sF_k$.

\subsection{A greedy tree-based estimator}\seclabel{greedy}

Suppose that $k$ is a power of two. This assumption can only, at
worst, smudge some constant factors in the final minimax rate. Using
the samples $X_1, \dots, X_n \iid f \in \sF_k$, we recursively
construct a rooted ordered binary tree $\widehat{T}$ which determines
a partition of the interval $\{1, \dots, k\}$, from which we can build
a histogram estimate $\hat{f}_n$ for $f$. Specifically, let $\rho$ be
the root of $\widehat{T}$, where $I_\rho = \{1, \dots, k\}$. We say
that $\rho$ covers the interval $I_\rho$. Then, for every node $u$ in
$\widehat{T}$ covering the interval
\[
  I_u = \{a_u, a_u + 1, \dots, a_u + |I_u| - 1\} ,
\]
we first check if $|I_u| = 1$, and if so we make $u$ a leaf in
$\widehat{T}$. Otherwise, if
\begin{align*}
  I_v &= \left\{a_u, a_u + 1, \dots, a_u + |I_u|/2 - 1\right\} , \\
  I_w &= I_u \setminus I_v
\end{align*}
are the first and second halves of $I_u$, we verify the condition
\begin{align}
  |N_v - N_w| > \sqrt{N_v + N_w} , \eqlabel{greedy-splitting-rule}
\end{align}
where $N_v, N_w$ are the number of samples which fall into the
intervals $I_v, I_w$, \ie
\[
  N_z = \sum_{i = 1}^n \mathbf{1}\{X_i \in I_z\}, \qquad \text{for } z \in \{v, w\} .
\]
The inequality \eqref{greedy-splitting-rule} is referred to as the
\emph{greedy splitting rule}. If \eqref{greedy-splitting-rule} is
satisfied, then create nodes $v, w$ covering $I_v$ and $I_w$
respectively, and add them to $\widehat{T}$ as left and right children
of $u$. If not, make $u$ a leaf in $\widehat{T}$.

\begin{figure}
  \centering
  \includegraphics[width=0.85\textwidth]{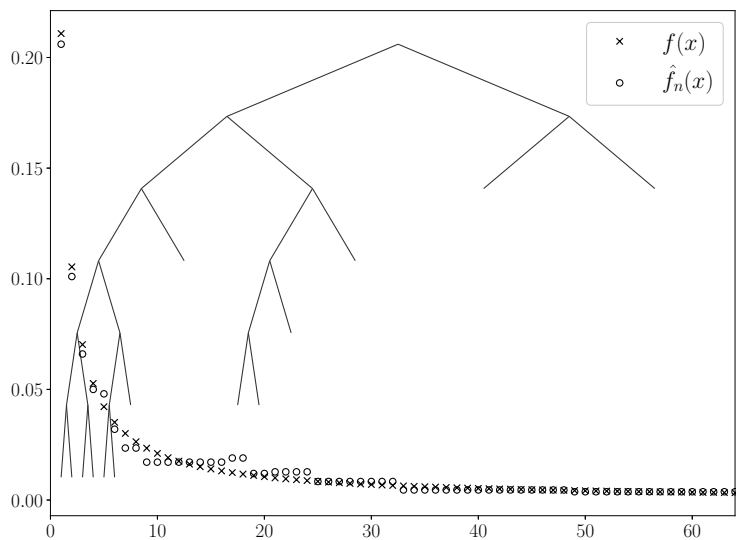}
  \caption{One sample of $\hat{f}_n$ for $n = 1000$ and $k = 64$,
    where $f(x) = 1/(x H_k)$ for $H_k$ the $k$-th Harmonic number. The
    tree $\widehat{T}$ is overlayed.}
  \figlabel{greedy}
\end{figure}
After applying this procedure, one obtains a (random) tree
$\widehat{T}$ with leaves $\widehat{L}$, and the set
$\{I_u \colon u \in \widehat{L}\}$ forms a partition of the support
$\{1, \dots, k\}$. Let $\hat{f}_n$ be the histogram estimate based on
this partition, \ie
\[
  \hat{f}_n(x) = \frac{N_u}{n |I_u|} , \qquad \text{if $x \in I_u, u \in \widehat{L}$.}
\]
The density estimate $\hat{f}_n$ is called the \emph{greedy tree-based
  estimator}. See \figref{greedy} for a typical plot of $\hat{f}_n$,
and a visualization of the tree $\widehat{T}$.

\begin{rem}
Intuitively, we justify the rule \eqref{greedy-splitting-rule} as
follows: We expect that $N_v$ is at least as large as $N_w$ by
monotonicity of the density $f$, and the larger the difference
$|N_v - N_w|$, the finer a partition around $I_v$ and $I_w$ should be
to minimize the error of a piecewise constant estimate of
$f$. However, even if $N_v$ and $N_w$ were equal in expectation, we
expect with positive probability that $N_v$ may deviate from $N_w$ on
the order of a standard deviation, \ie on the order of
$\sqrt{N_v + N_w}$, and this determines the threshold for splitting.
\end{rem}

\begin{rem}\remlabel{birge-op}
  One could argue that any good estimate of a non-increasing density
  should itself be non-increasing, and the estimate $\hat{f}_n$ does
  not have this property. This can be rectified using a method of
  Birg\'{e}~\cite{birge-risk}, who described a transformation of
  piecewise-constant density estimates which does not increase risk
  with respect to non-increasing densities. Specifically, suppose that
  the estimate $\hat{f}_n$ is not non-increasing. Then, there are
  consecutive intervals $I_v$, $I_w$ such that $\hat{f}_n$ has
  constant value $y_v$ on $I_v$ and $y_w$ on $I_w$, and $y_v <
  y_w$. Let the transformed estimate be constant on $I_v \cup I_w$,
  with value
  \[
    \frac{y_v |I_v| + y_w |I_w|}{|I_v| + |I_w|} ,
  \]
  \ie the average value of $\hat{f}_n$ on $I_v \cup I_w$. Iterate the
  above transformation until a non-increasing estimate is obtained. It
  can be proven that this results in a unique estimate $\hat{f}_n'$,
  regardless of the order of merged intervals, and that
  \[
    \TV(\hat{f}_n', f) \le \TV(\hat{f}_n, f) .
  \]
\end{rem}

\subsection{An idealized tree-based estimator}\seclabel{ideal}

\begin{figure}
  \centering
  \includegraphics[width=0.85\textwidth]{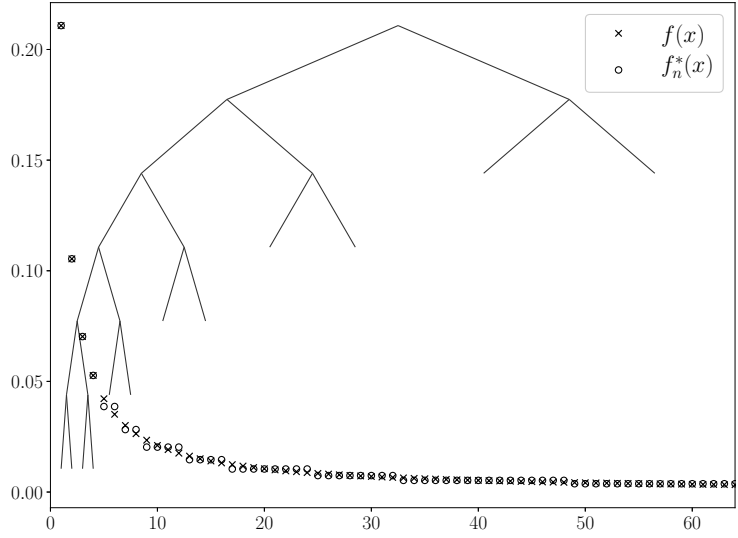}
  \caption{A plot of $f^*_n$ for $n = 1000$ and $k = 64$, where
    $f(x) = 1/(x H_k)$. The tree $T^*$ is overlayed.}
  \figlabel{ideal}
\end{figure}
Instead of analyzing the greedy tree-based estimator $\hat{f}_n$ of
the preceding section, we fully analyze an idealized version.
Indeed, in \eqref{greedy-splitting-rule}, the quantities $N_z$ are
distributed as $\Bin(n, f_z)$ for $z \in \{v, w\}$, where we define
\[
  f_z = \sum_{x \in I_z} f(x) .
\]
If we replace the quantities in \eqref{greedy-splitting-rule} with
their expectations, we obtain the \emph{idealized splitting rule}
\begin{align}
  f_v - f_w > \sqrt{\frac{f_v + f_w}{n}} , \eqlabel{id-splitting-rule}
\end{align}
where we note that $f_v \ge f_w$, since $f$ is non-increasing. Using
the same procedure as in the preceding section, replacing the
splitting rule with \eqref{id-splitting-rule}, we obtain a
deterministic tree $T^* = T^*(f)$ with leaves $L^*$, and we set
\[
  f^*_n(x) = \bar{f}_u = \frac{f_u}{|I_u|} , \qquad \text{if } x \in I_u, \, u \in L^* ,
\]
\ie $f^*_n$ is constant and equal to the average value of $f$ on each
interval $I_u$ for $u \in L^*$. We call $f^*_n$ the \emph{idealized
  tree-based estimate}. See \figref{ideal} for a visualization of
$f^*_n$ and $T^*$. It may be instructive to compare \figref{ideal} to
\figref{greedy}.

Of course, $T^*$ and $f^*_n$ both depend
intimately upon knowledge of the density $f$; in practice, we only
have access to the samples $X_1, \dots, X_n \iid f$, and the density
$f$ itself is unknown. In particular, we cannot practically use
$f^*_n$ as an estimate for unknown $f$. Importantly, as we will soon
show, we can still use $f^*_n$ along with
\corrref{minimum-distance-risk} to get a minimax rate upper bound for
$\sF_k$.

\begin{prop}\proplabel{id-tv}
  \[
    \TV(f^*_n, f) \le \frac{5}{2} \sqrt{\frac{|\{u \in L^* \colon |I_u| > 1\} |}{n}} .
  \]
\end{prop}
\begin{proof}
  Writing out the TV-distance explicitly, we have
  \begin{align*}
    \TV(f^*_n, f) &= \frac{1}{2} \sum_{x \in \{1, \dots, k\}} |f^*_n(x) - f(x)| \\
                  &= \frac{1}{2} \sum_{u \in L^*} \sum_{x \in I_u} |f^*_n(x) - f(x)| \\
                  &= \frac{1}{2} \sum_{u \in L^*} \sum_{x \in I_u} | \bar{f}_u - f(x) | .
  \end{align*}
  \begin{figure} 
    \centering
    \includegraphics{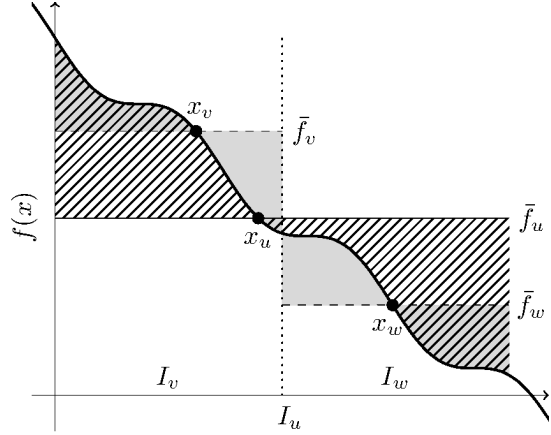}
    \caption{A visualization of the $L_1$ distance between $f^*_n$ and
      $f$ on $I_u$.}
    \figlabel{biasarea}
  \end{figure}
  
  Let $u \in L^*$, and define
  $A_u = \sum_{x \in I_u} |\bar{f}_u - f(x)|$. If $|I_u| = 1$, then
  $A_u = 0$, so assume that $|I_u| > 1$. In this case, let $I_v$ and
  $I_w$ be the left and right halves of the interval $I_u$, and let
  $\bar{f}_v$ and $\bar{f}_w$ be the average value of $f$ on $I_v$ and
  $I_w$ respectively. Write also
  \[
    B_v = \sum_{x \in I_v} |\bar{f}_v - f(x)| , \qquad   B_w = \sum_{x \in I_w} |\bar{f}_w - f(x)| .
  \]
  Refer to \figref{biasarea}. We view $A_u$ as the positive area
  between the curve $f$ and the line $\bar{f}_u$; in the figure, this
  is the patterned area. Then, $B_v$ is the positive area between $f$
  and $\bar{f}_v$ on $I_v$, which is represented as the gray area on
  $I_v$ in \figref{biasarea}, and $B_w$ is the positive area between
  $f$ and $\bar{f}_w$ on $I_w$, the gray area on $I_w$ in
  \figref{biasarea}. For $z \in \{v, u, w\}$, let $x_z$ be the largest
  point in $I_z$ for which $f(x_z) \ge \bar{f}_z$.  By the triangle
  inequality,
  \begin{align*}
    A_u &\le (\bar{f}_v - \bar{f}_u) |I_v| + (\bar{f}_u - \bar{f}_w) |I_w| + B_v + B_w \\
        &= (f_v - f_w) + B_v + B_w .
  \end{align*}
  Furthermore,
  \begin{align*}
    B_v &= \sum_{x \in I_v, \, x \le x_v} (f(x) - \bar{f}_v) + \sum_{x \in I_v, \, x > x_v} (\bar{f}_v - f(x)) \\
        &= 2 \sum_{x \in I_v, \, x > x_v} (\bar{f}_v - f(x)) \\
        &\le 2 |I_v| (\bar{f}_v - \bar{f}_w) \\
        &= 2 (f_v - f_w) ,
  \end{align*}
  where the second equality follows by the choice of $x_v$. A similar
  relation holds for $B_w$, whence
  \[
    A_u \le 5 (f_v - f_w) \le 5 \sqrt{f_u/n} ,
  \]
  where this last inequality follows from the splitting rule
  \eqref{id-splitting-rule}, since $u \in L^*$ and $|I_u| > 1$. So,
  \begin{align*}
    \TV(f^*_n, f) &= \frac{1}{2} \sum_{u \in L^*} A_u \\
                  &\le \frac{5}{2 \sqrt{n}} \sum_{u \in L^* \colon |I_u| > 1} \sqrt{f_u} \\
                  &\le \frac{5}{2} \sqrt{\frac{|\{u \in L^* \colon |I_u| > 1\}|}{n}} ,
  \end{align*}
  by the Cauchy-Schwarz inequality.
\end{proof}

\begin{prop}\proplabel{id-num-leaves}
  If $n \ge 64$ and $2 n^{1/3} \le k < n^{1/3} 2^n$, then
  \[
    |L^*| \le 12 n^{1/3} {\left( \log_2 (k/n^{1/3}) \right)}^{2/3} .
  \]
\end{prop}
\begin{proof}
  Note that $T^*$ has height at most $\log_2 k$. Let $U_j$ be the set
  of nodes at depth $j - 1$ in $T^*$ which have at least one leaf as a
  child, for $1 \le j \le \log_2 k$, and label the children of the
  nodes in $U_j$ in order of appeareance from right to left in $T^*$
  as $u_1, u_2, \dots, u_{2 |U_j|}$. Since none of the nodes in $U_j$
  are themselves leaves, then by \eqref{id-splitting-rule},
  \[
    f_{u_2} - f_{u_1} > \sqrt{\frac{f_{u_1} + f_{u_2}}{n} } ,
  \]
  and in particular since $f_{u_1} \ge 0$, then
  $f_{u_2} > \sqrt{f_{u_2}/n},$ so that $f_{u_2} > 1/n$. In general,
  \[
    f_{u_{2i}} > f_{u_{2i - 2}} + \sqrt{ \frac{2 f_{u_{2i - 2}}}{n} }, 
  \]
  and this recurrence relation can be solved to obtain that
  \begin{align}
    f_{u_{2i}} \ge \frac{i^2}{n} . \eqlabel{monotone-leaf-prob}
  \end{align}
  Let $L_j$ be the set of leaves at level $j$ in $T^*$. The leaves at
  level $j$ in order from right to left form a subsequence
  $v_1, v_2, \dots, v_{|L_j|}$ of $u_1, u_2, \dots, u_{2|U_j|}$. Write
  $q_j$ for the total probability mass of $f$ held in the leaves
  $L_j$, \ie
  \[
    q_j = \sum_{v \in L_j} f_v = \sum_{i = 1}^{|L_j|} f_{v_i} .
  \]
  By \eqref{monotone-leaf-prob} and since $f_{v_i} \ge f_{u_i}$ for
  each $i$,
  \begin{align*}
    \sum_{i = 1}^{|L_j|} f_{v_i} \ge \sum_{i = 1}^{\floor{|L_j|/2}} f_{u_{2i}} 
                                 \ge \sum_{i = 1}^{\floor{|L_j|/2}} \frac{i^2}{n} 
                                 \ge \frac{(\floor{|L_j|/2})^3}{3 n} ,
  \end{align*}
  so that
  \[
    |L_j| \le 2 + 2(3 n q_j)^{1/3} \le 2 + 3 (n q_j)^{1/3} .
  \]
  Summing over all leaves and using the facts that $n \ge 64$ and
  $2n^{1/3} \le k < n^{1/3} 2^n$,
  \begin{align*}
    |L^*| &= \sum_{j = 0}^{\floor{(1/3) \log_2 n} - 1} |L_j| + \sum_{j = \floor{(1/3) \log_2 n}}^{\log_2 k} |L_j| \\
          &\le n^{1/3} + \sum_{j = \floor{(1/3) \log_2 n}}^{\log_2 k} (2 + 3 (n q_j)^{1/3}) \\
          &\le n^{1/3} + 4 \log_2 (k/n^{1/3}) + 3 n^{1/3} \sum_{j = \floor{(1/3) \log_2 n}}^{\log_2 k} q_{j}^{1/3} .
  \end{align*}
  By H\"{o}lder's inequality,
  \begin{align*}
    \sum_{j = \floor{(1/3) \log_2 n}}^{\log_2 k} q_{j}^{1/3} &\le {\left( \sum_{j = 1}^{\log_2 k} q_j \right)}^{1/3} {\left( \sum_{j = \floor{(1/3) \log_2 n}}^{\log_2 k} 1 \right)}^{2/3} \\
                                                                  &\le {\left(3 \log_2 (k/n^{1/3}) \right)}^{2/3} , 
  \end{align*}
  so finally
  \begin{align*}
    |L^*| &\le n^{1/3} + 4 \log_2 (k/n^{1/3}) + 7 n^{1/3} {\left( \log_2 (k/n^{1/3}) \right)}^{2/3} \\
        &\le 12 n^{1/3} {\left( \log_2 (k/n^{1/3}) \right)}^{2/3} . \qedhere
  \end{align*}
\end{proof}

\begin{proof}[Proof of the upper bound in \thmref{monotone-result}]
  The case $k \ge n^{1/3} 2^n$ is trivial, and follows simply because
  the TV-distance is always upper bounded by $1$.

  Suppose next that $2 n^{1/3} > k$. In this regime, we can use a
  histogram estimator for $f$ with bins of size $1$ for each element
  of $\{1, \dots, k\}$. It is well known that risk of this estimator
  is on the order of $\sqrt{k/n}$~\cite{comb-methods}.

  Finally, suppose that $2 n^{1/3} \le k < n^{1/3} 2^n$. Let
  $\sF_\Theta$ be the class of all piecewise-constant probability
  densities on $\{1, \dots, k\}$ which have $\ell = |L^*|$ parts; in
  particular, $f^*_n \in \sF_\Theta$. Let $\sA_\Theta$ be the Yatracos
  class of $\sF_\Theta$,
  \[
    \sA_\Theta = \Bigl\{ \{x \colon f(x) > g(x)\} \colon f \neq g \in \sF_\Theta\Bigr\} .
  \]
  Then, $\sA_\Theta \subseteq \sA$, where $\sA$ is the class of all
  unions of at most $\ell$ intervals in $\N$. It is well known that
  $\VC(\sA) = 2\ell$, so $\VC(\sA_\Theta) \le 2\ell$. By
  \corrref{minimum-distance-risk} and \propref{id-tv}, there are
  universal constants $c_1, c_2 > 0$ for which
  \begin{align*}
    \sR_n(\sF_k) &\le 3 \sup_{f \in \sF_k} \inf_{\theta \in \Theta} \TV(f_{n, \theta}, f) + c_1 \sqrt{\ell/n} \\
                 &\le 3 \sup_{f \in \sF_k} \TV(f^*_n, f) + c_1 \sqrt{\ell/n} \\
                 &\le c_2 \sqrt{\ell/n} .
  \end{align*}
  By \propref{id-num-leaves}, we see that for sufficiently large $n$,
  there is a universal constant $c_3 > 0$ such that
  \[
    \sR_n(\sF_k) \le c_3 {\left( \frac{\log (k/n^{1/3})}{n} \right)}^{1/3} . \qedhere
  \]
\end{proof}

\begin{rem}\remlabel{cont-monotone}
  Fix $B > 0$ and let $\sF_B'$ be the class of all non-increasing
  densities supported on $[0, 1]$ and bounded from above by $B$. Our
  method can be applied to prove a minimax rate upper bound
  $\sF_B'$. Now, the tree $T^*$ underlying the idealized tree-based
  estimator is truncated at some given level, say $m$ to be specified,
  and the idealized estimator should take on the average value of the
  true density $f$ on the truncated leaves. Write $d(u)$ for the depth
  of the node $u$ in $T^*$. As in \propref{id-tv},
  \begin{align*}
    &\TV(f_n^*, f) = \frac{1}{2} \sum_{u \in L^*} A_u \\
    & \qquad\le \frac{5}{2 \sqrt{n}} \sum_{u \in L^*, \, d(u) < m} \sqrt{f_u} + \frac{1}{2} \sum_{u \in L^*, \, d(u) = m} A_u .
  \end{align*}
  The argument of \propref{id-num-leaves} allows us to control the
  first sum, so that for some universal constant $c_1 > 0$,
  \[
    \frac{5}{2\sqrt{n}} \sum_{u \in L^*, \, d(u) < m} \sqrt{f_u} \le c_1 {\left( \frac{m}{n}\right)}^{1/3} .
  \]
  On the other hand, since $A_u \le 5(f_v - f_w)$ for $v$ the left
  child and $w$ the right child of $u$, then
  \[
    \frac{1}{2} \sum_{u \in L^*, \, d(u) = m} A_u \le \frac{5 B}{2^{m + 2}} .
  \]
  An optimal choice of $m$ has that for a universal constant
  $c_2 > 0$,
  \[
    \TV(f_n^*, f) \le c_2 {\left(\frac{\log_2 (B n^{1/3})}{n}\right)}^{1/3} .
  \]
  From here, using the same method as in the proof of
  \thmref{monotone-result}, it follows that for some universal
  $c_3 > 0$,
  \[
    \sR_n(\sF_B') \le c_3 {\left(\frac{\log_2 (B
          n^{1/3})}{n}\right)}^{1/3} .
  \]
\end{rem}

%% file: convex.tex
\section{Non-increasing convex densities}\seclabel{convex}

Recall that $\sG_k$ is the class of non-increasing convex densities
supported on $\{1, \dots, k\}$. Then, $\sG_k$ forms a subclass of
$\sF_k$, which we considered in \secref{monotone}. This section is
devoted to extending the techniques of \secref{monotone} in order to
obtain a minimax rate upper bound on $\sG_k$. Again, the lower bound
is proved using standard techniques in \appref{convex-lower}.

In this section, we assume that $k$ is a power of three. In order to
prove the upper bound of \thmref{convex-result}, we construct a
ternary tree just as in \secref{monotone}, now with a ternary
splitting rule, where if a node $u$ has children $v, w, r$ in order
from left to right, we split and recurse if
\begin{equation}
  f_v - 2 f_w + f_r > \sqrt{\frac{f_v + f_w + f_r}{n}} . \eqlabel{convex-splitting-rule}
\end{equation}
Here we obtain a tree $T^\dagger = T^\dagger(f)$ with leaves
$L^\dagger$. If $u \in L^\dagger$ has children $v, w, r$ from left to
right, let $m_z$ be the midpoint of $I_z$ for $z \in \{v, w, r\}$. Let
the estimate $f^\dagger_n$ on $I_u$ be the line passing through the
points $(m_v, \bar{f}_v)$ and $(m_r, \bar{f}_r)$. Again, if
$|I_u| = 1$, then $f^\dagger_n(x) = f(x)$. We refer to $f_n^\dagger$
as the \emph{idealized tree-based estimate} for $f$.
\begin{rem}
  Since $f$ is non-increasing, the operation of \remref{birge-op} can
  again by applied to $f_n^\dagger$ to obtain a non-increasing
  estimate ${f_n^\dagger}'$ for which
  \[
    \TV({f_n^\dagger}', f) \le \TV(f_n^\dagger, f) .
  \]
\end{rem}

\begin{prop}\proplabel{conv-id-tv}
  \[
    \TV(f_n^\dagger, f) \le \frac{41}{48} \sqrt{\frac{|\{u \in L^\dagger \colon |I_u| > 1\}|}{n} } .
  \]
\end{prop}
Before proving this, we first note that by convexity of $f$, the slope
of the line passing through $(m_w, \bar{f}_w)$ and $(m_r, \bar{f}_r)$
is at least the slope of the line passing through $(m_v, \bar{f}_v)$
and $(m_w, \bar{f}_w)$. Equivalently,
\[
  f_r - f_w \ge f_w - f_v \iff f_v - 2 f_w + f_r \ge 0 .
\]
\begin{proof}[Proof of \propref{conv-id-tv}]
  \begin{figure} 
    \centering
    \includegraphics{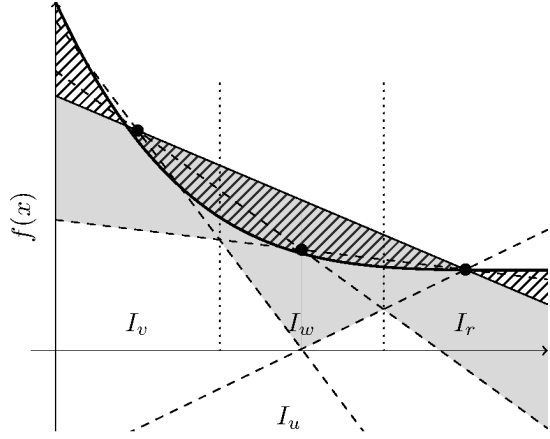}
    \caption{A visualization of the $L_1$ distance between $f^\dagger_n$ and
      $f$ on $I_u$.}
    \figlabel{convexarea}
  \end{figure}  
  As in the proof of \propref{id-tv}, we have
  \[
    \TV(f_n^\dagger, f) = \frac{1}{2} \sum_{u \in L^\dagger \colon |I_u| > 1} A_u ,
  \]
  for $A_u = \sum_{x \in I_u} |f_n^\dagger(x) - f(x)|$. We refer to
  \figref{convexarea} for a visualization of the quantity $A_u$, which
  is depicted as the patterned area.


  For $z \in \{v, w, r\}$, write
  \[
    B_z = \sum_{x \in I_z} |f^\dagger_n(x) - f(x)| .
  \]
  By convexity of $f$,
  \[
    B_w = \sum_{x \in I_w} (f^\dagger_n(x) - f(x)) .
  \]
  Observe also by convexity that $f(m_w) \ge \bar{f}_w$ and
  $f(m_r) \ge \bar{f}_r$. So, the line segment between
  $(m_w, \bar{f}_w)$ and $(m_r, \bar{f}_r)$ lies above $f$. Let
  $g_{wr} \colon \R \to \R$ be the line passing through the points
  $(m_w, \bar{f}_w)$ and $(m_r, \bar{f}_r)$. Then,
  \[
    \int_{I_w} g_{wr} = \bar{f}_w |I_w| = f_w ,
  \]
  Let $m_{wr}$ be the midpoint of $m_w$ and $m_r$, and $m_{vw}$ be the
  midpoint of $m_v$ and $m_w$. Let $x_w \in [-\infty, m_{wr}]$ be the
  leftmost point where $g_{wr}$ intersects $f$, if at all. Then,
  \[
    \int_{I_w \cap (-\infty, x_w]} (f - g_{wr}) = \int_{I_w \cap (x_w, \infty)} (g_{wr} - f),
  \]
  Since the right-hand side is non-negative, then indeed $x_w \in I_w$
  and it must be that $f$ lies above $g_{wr}$ to the left of
  $x_w$. Similarly, if $x_w' \in [m_{vw}, \infty]$ denotes the
  rightmost point where the line $g_{vw} \colon \R \to \R$ passing
  through $(m_v, \bar{f}_v)$ and $(m_w, \bar{f}_w)$ intersects $f$,
  then $x_w' \in I_w$ and $f$ lies above $g_{vw}$ to the right of
  $x_w'$. Therefore,
  \begin{align*}
    B_w &\le \int_{I_w \cap (-\infty, m_w]} (f^\dagger_n - g_{vw}) + \int_{I_w \cap (m_w, \infty)} (f^\dagger_n - g_{wr}) \\
        &= 3(f_v - 2 f_w + f_r)/8 .
  \end{align*}

  It remains to bound $B_v$ and $B_r$. Let $x_v \in I_v$ be the point
  where the line passing through $(m_v, \bar{f}_v)$ and
  $(m_r, \bar{f}_r)$ intersects $f$. As before, this points exists,
  and since $f$ is non-increasing, $x_v \le m_v$. Futhermore,
  \begin{align*}
    B_v &= \int_{I_v \cap (-\infty, x_v]} (f - f^\dagger_n) + \int_{I_v \cap (x_v, \infty)} (f^\dagger_n - f) \\
        &= 2 \int_{I_v \cap (x_v, \infty)} (f^\dagger_n - f) \\
        &\le 2 \int_{I_v} (f^\dagger_n - g_{wr}) \\
        &= 2(f_v - 2f_w + f_r)/3 ,
  \end{align*}
  where the inequality follows from convexity and earlier remarks. A
  similar argument follows for $B_r$.

  In total,
  \[
    A_u = B_v + B_w + B_r \le 41(f_v - 2 f_w + f_r)/24 .
  \]
  The result then follows from the splitting rule
  \eqref{convex-splitting-rule} and the Cauchy-Schwarz inequality.
\end{proof}

\begin{prop}
  If $n \ge 3^{10}$ and $3 n^{1/5} \le k < n^{1/5} 3^n$, then
  \[
    |L^\dagger| \le 34 n^{1/5} {\left( \log_3 (k/n^{1/5})\right)}^{4/5} .
  \]
\end{prop}
\begin{proof}
  The tree $T^\dagger$ has height at most $\log_3 k$. Let $U_j$ be the
  set of nodes at depth $j - 1$ in $T^\dagger$ with at least one leaf
  as a child, for $1 \le j \le \log_3 k$, labelled in order of
  appearance from right to left in $T^\dagger$ as
  $u_1, u_2, \dots, u_{3 |U_j|}$.
  By the convex splitting rule \eqref{convex-splitting-rule}, and
  since $f$ is non-increasing,
  \[
    f_{u_3} - f_{u_2} > f_{u_2} - f_{u_1} + \sqrt{\frac{f_{u_1} + f_{u_2} +
        f_{u_3}}{n}} \ge \sqrt{\frac{f_{u_3} - f_{u_2}}{n}} ,
  \]
  so in particular, $f_{u_3} - f_{u_2} > 1/n$, and $f_{u_3} > 1/n$. In
  general,
  \begin{align}
    f_{u_{3i}} - f_{u_{3i - 1}} &> f_{u_{3(i - 1)}} - f_{u_{3(i - 1) - 1}} + \sqrt{ \frac{3 f_{u_{3(i - 1)}}}{n}} \notag \\
                                &\ge f_{u_{3(i - 1)}} - f_{u_{3(i - 1) - 1}} + \sqrt{\frac{3 \sum_{j = 1}^{i - 1} (f_{u_{3j}} - f_{u_{3j - 1}})}{n}} . \eqlabel{concave-induction}
  \end{align}
  We claim now that $f_{u_{3i}} - f_{u_{3i - 1}} > \frac{i^3}{27n}$,
  which we prove by induction; the base case is shown above, and by
  the induction hypothesis,
  \begin{align*}
    \eqref{concave-induction} &\ge \frac{(i - 1)^3}{27n} + \sqrt{\frac{3 \sum_{j = 1}^{i - 1} (j^3/27n)}{n}} \\
                              &\ge \frac{(i - 1)^3}{27n} + \frac{(i - 1)^2}{6n} \\
                              &\ge \frac{i^3}{27n} ,
  \end{align*}
  for all $i \ge 4$, while the cases $i = 2, 3$ can be manually
  verified. Then, by monotonicity of $f$,
  \begin{align}
    f_{u_{3i}} &> \frac{i^3}{27n} + f_{u_{3i - 1}} \notag \\
               &\ge \frac{i^3}{27n} + f_{u_{3(i - 1)}} \notag \\
               &\ge \sum_{j = 1}^i \frac{j^3}{27n} \notag \\
               &\ge \frac{i^4}{108 n} . \eqlabel{concave-prob-lower}
  \end{align}
  Let now $L_j$ be the set of leaves at level $j$ in $T^\dagger$. The
  leaves at level $j$ in order from right to left form a subsequence
  $v_1, \dots, v_{|L_{j}|}$ of $u_1, \dots, u_{3|U_j|}$. Let $q_j$ be
  the total probability mass of $f$ held in the leaves $L_j$. By
  \eqref{concave-prob-lower} and since $f_{v_i} \ge f_{u_i}$ for each
  $i$,
  \begin{align*}
    q_j \ge \sum_{i = 1}^{\floor{|L_j|/3}} f_{u_{3i}} \ge \sum_{i = 1}^{\floor{|L_j|/3}} \frac{i^4}{108n} \ge \frac{(\floor{|L_j|/3})^5}{540 n} ,
  \end{align*}
  so that
  \[
    |L_j| \le 3 + 3 (540 n q_j)^{1/5} \le 3 + 11 (n q_j)^{1/5} .
  \]
  Summing over all leaves,
  \begin{align*}
    |L^\dagger|  &\le n^{1/5} + \sum_{j = \floor{(1/5) \log_3 n}}^{\log_3 k} |L_j| \\
                 &\le n^{1/5} + 6 \log_3 (k/n^{1/5}) + 11 n^{1/5} \sum_{j = \floor{(1/5) \log_3 n}}^{\log_3 k} q_j^{1/5} .
  \end{align*}
  By H\"{o}lder's inequality,
  \begin{align*}
    \sum_{j = \floor{(1/5) \log_3 n}}^{\log_3 k} q_j^{1/5} &\le {\left( \sum_{j = 0}^{\log_3 k} q_j \right)}^{1/5} {\left( \sum_{j = \floor{(1/5) \log_3 k}}^{\log_3 k} 1 \right)}^{4/5} \\
                                                               &\le {\left(3 \log_3 (k/n^{1/5}) \right)}^{4/5} ,
  \end{align*}
  so finally
  \begin{align*}
    |L^\dagger| &\le n^{1/5} + 6 \log_3 (k/n^{1/5}) + 27 n^{1/5} {\left(\log_3 (k/n^{1/5})\right)}^{4/5} \\
                &\le 34 n^{1/5} {\left(\log_3 (k/n^{1/5}) \right)}^{4/5} . \qedhere
  \end{align*}
\end{proof}

\begin{proof}[Proof of the upper bound in \thmref{convex-result}]
  The proof is similar to that of \thmref{monotone-result}.
\end{proof}

\begin{rem}
  As in \remref{cont-monotone}, the argument can replicated in the
  continuous case, for bounded non-increasing convex densities
  supported on $[0, 1]$.
\end{rem}

%% file: discussion.tex
\section{Discussion}\seclabel{conclusion}

It seems likely, given our results on the idealized tree-based
estimators from \secref{ideal} and \secref{convex}, that the greedy
tree-based estimators also behave well. In particular, we suspect that
our greedy tree-based estimators are minimax-optimal within
logarithmic factors. We leave this open to future work.

It is also often desirable for nonparametric estimators to be
adaptive, in the sense that they attain the optimal minimax rate
without depending on some of the important features of the
nonparametric class in question. In some cases, an adaptive density
estimate can be constructed by first estimating these features, and
then building a density estimate assuming the estimated features. For
example, in \cite{birge-risk}, an adapative estimate for
non-increasing densities is developed by first estimating the size of
the support, and plugging this estimated support size into a
non-adaptive estimate. We expect that in this manner, our method can
be made adaptive.

The techniques of this paper seem to naturally extend to higher
dimensions. Take, for instance, the class of block-decreasing
densities, whose minimax rate was identified by Biau and
Devroye~\cite{block}. This is the class of densities supported on
$[0, 1]^d$ bounded by some constant $B \ge 0$, such that each density
is non-increasing in each coordinate if all other coordinates are held
fixed. The discrete version of this class has each density supported
on $\{1, \dots, k\}^d$, with the monotonicity constraint. In order to
estimate such a density, one could devise an oriented binary splitting
rule analogous to \eqref{id-splitting-rule} and carry out a similar
analysis as performed in \secref{ideal}.

Furthermore, we expect that there are many other classes of
one-dimensional densities whose optimal minimax rate could be
identified using our approach, like the class of $\ell$-monotone
densities on $\{1, \dots, k\}$, where a function $f$ is called
\emph{$\ell$-monotone} if it is non-negative and if $(-1)^j f^{(j)}$
is non-increasing and convex for all $j \in \{0, \dots, \ell - 2\}$ if
$\ell \ge 2$, and where $f$ is non-negative and non-increasing if
$\ell = 1$. This paper tackles the cases of $\ell = 1$ and $\ell =
2$. Write $\sF_{k, \ell}$ for the class of $\ell$-monotone densities
on $\{1, \dots, k\}$. See Balabdaoui and Wellner~\cite{k-monotone,
  k-monotone-2} for texts concerning the density estimation of
$\ell$-monotone densities. It seems likely that our method could be
applied to prove the following conjecture.
\begin{conj}
Let $f \colon \N \times \N \times \N \times \R \to \R$ be
\[
 f(n, k, \ell, \alpha) = \left\{
                         \begin{array}{ll}
                                \sqrt{k/n} & \mbox{if $2 \le k \le \alpha n^{\frac{1}{2\ell + 1}}$,} \\
                                {\left(\frac{\log_\alpha (k/n^{\frac{1}{2\ell + 1}})}{n}\right)}^{\frac{\ell}{2\ell + 1}} & \mbox{if $\alpha n^{\frac{1}{2\ell + 1}} \le k \le n^{\frac{1}{2\ell + 1}} \alpha^n$,} \\
                                1 & \mbox{if $n^{\frac{1}{2\ell + 1}} \alpha^n \le k.$}
                         \end{array}\right.
\]
Let $\ell \ge 1$ be fixed. There are constants $\alpha, C, n_0 \ge 1$
depending only on $\ell$ such that, for $n \ge n_0$,
\[
  \frac{1}{C} \le \frac{\sR_n(\sF_{k, \ell})}{f(n, k, \ell, \alpha)} \le C .
\]
\end{conj}
The main obstacle in proving the above would be the development of
good local estimates for $\ell$-monotone densities, in the same flavor
as \propref{id-tv} and \propref{conv-id-tv}.

Our approach also likely can be applied to the class of all
log-concave discrete distributions, where we recall that $f : \N \to
[0, 1]$ is called \emph{log-concave} if
\[
  f(x) f(x + 2) \le f(x + 1)^2 , \quad \text{for all } x \ge 1 .
\]
See \cite{kane-logconcave-up, samworth-lower, samworth-survey} for a
small selection of works on the density estimation of $d$-dimensional
log-concave continuous densities. The optimal Hellinger distance
minimax rate (within logarithmic factors) for this class was recently
obtained by Dagan and Kur~\cite{dagan}, who showed that it is attained
by the maximum-likelihood estimate. There remains a small gap between
the best known upper and lower bounds in the TV-distance minimax rate
as of the time of writing.

\section*{Acknowledgments}

We would like to thank the three reviewers and an associate editor for
their helpful comments and suggestions.

%% file: appendix.tex
\section{Lower bounds}\applabel{app-1}

\begin{lem}[Assouad's Lemma~\cite{assouad, comb-methods}]\lemlabel{assouad}
  Let $\sF$ be a class of densities supported on the set $\sX$. Let
  $A_0, A_1, \dots, A_r$ be a partition of $\sX$, and
  $g_{ij} \colon A_i \to \R$ for $0 \le i \le r$ and $j \in \{0, 1\}$
  be some collection of functions. For
  $\theta = (\theta_1, \dots, \theta_r) \in \{0, 1\}^r$, define the
  function $f_\theta \colon \sX \to \R$ by
  \[
    f_\theta(x) = \left\{\begin{array}{ll}
                           g_{00}(x) & \mbox{if $x \in A_0$,} \\
                           g_{i \theta_i}(x) & \mbox{if $x \in A_i$,}
                    \end{array} \right.
  \]
  such that each $f_\theta$ is a density on $\sX$. Let
  $\zeta_i \in \{0, 1\}^n$ agree with $\theta$ on all bits except for
  the $i$-th bit. Then, suppose that
  \[
    0 < \beta \le \inf_\theta \inf_{1 \le i \le r} \int \sqrt{f_\theta f_{\zeta_i}} ,
  \]
  and
  \[
    0 < \alpha \le \inf_\theta \inf_{1 \le i \le r} \int_{A_i} |f_\theta - f_{\zeta_i}| .
  \]
  Let $\sH$ be the hypercube of densities
  \[
    \sH = \{f_\theta \colon \theta \in \{0, 1\}^r\} .
  \]
  If $\sH \subseteq \sF$, then
  \[
    \sR_n(\sF) \ge \frac{r \alpha}{4} \left(1 - \sqrt{2 n (1 - \beta)}\right) .
  \]
\end{lem}

\subsection{Proof of the lower bound in \thmref{monotone-result}.}\applabel{monotone-lower}

Suppose first that $e^8 n^{1/3} \le k \le n^{1/3} e^n$. Let
$A_1, \dots, A_r$ be consecutive intervals of even cardinality,
starting from the leftmost atom $1$. Split each $A_i$ in two equal
parts, $A_i'$ and $A_i''$. Let $\eps \in (0, 1/\sqrt{2})$, and set
\begin{align*}
  g_{i 0}(x) &= \left\{\begin{array}{ll}
                        \frac{1 + \eps}{r |A_i|} & \mbox{if $x \in A_i'$,} \\
                        \frac{1 - \eps}{r |A_i|} & \mbox{if $x \in A_i''$,}
                      \end{array} \right. \\
  g_{i 1}(x) &= \frac{1}{r |A_i|} .
\end{align*}
It is clear that each $f_\theta$ is a density. In order for each
$f_\theta$ to be monotone, we require that
\[
  \frac{1 - \eps}{|A_i|} \ge \frac{1 + \eps}{|A_{i + 1}|} ,
\]
and in particular
\[
  |A_i| \ge |A_1| \left(\frac{1 + \eps}{1 - \eps}\right)^{i - 1} .
\]
Pick $|A_1| = 2$. Since $\log (1 + \eps) - \log (1 - \eps) \le 4 \eps$
for $\eps \in (0, 1/\sqrt{2})$, it suffices to take
\[
  |A_i| \ge a_i = 2 e^{4 \eps(i - 1)} .
\]
Let $|A_i|$ be the smallest even integer at least equal to $a_i$, so
that $a_i \le |A_i| \le a_i + 2$, and thus
\[
  \sum_{i = 1}^r |A_i| \le 2r + \frac{2 e^{4\eps r}}{e^{4\eps} - 1} \le 2r + \frac{e^{4 \eps r}}{2 \eps} .
\]
Since the support of our densities is $\{1, \dots, k\}$, then we ask
that this last upper bound not exceed $k$. We can guarantee this in
particular with a choice of $r$ and $\eps$ for which
\begin{equation}
  2r \le \frac{k}{2}, \quad \text{and} \quad \frac{e^{4 \eps r}}{2\eps} \le \frac{k}{2} . \eqlabel{monotone-cond}
\end{equation}
Fix $1 \le i \le r$. Then,
\begin{align*}
  \int \left(\sqrt{f_\theta} - \sqrt{f_{\zeta_i}}\right)^2 &= \sum_{x \in A_i} \left(\sqrt{f_\theta(x)} - \sqrt{f_{\zeta_i}(x)}\right)^2 \\
                                                &= \frac{2}{r} - \frac{1}{r}\left(\sqrt{1 + \eps} + \sqrt{1 - \eps}\right) \\
                                                &\le \frac{\eps^2}{r} ,
\end{align*}
so
\begin{align*}
  \int \sqrt{f_\theta f_{\zeta_i}} &= 1 - \frac{1}{2} \int \left(\sqrt{f_\theta} - \sqrt{f_{\zeta_i}}\right)^2 \\
                                   &\ge 1 - \frac{\eps^2}{2r} .
\end{align*}
On the other hand,
\[
  \int_{A_i} |f_\theta - f_{\zeta_i}| = \sum_{x \in A_i} |f_\theta(x) - f_{\zeta_i}(x)| = \frac{\eps}{r} .
\]
Now pick
\[
  \eps = \frac{1}{4} \left(\frac{\log(k/n^{1/3})}{n}\right)^{1/3} ,
\]
and $r$ for which
\[
  \sqrt{\frac{n \eps^2}{r}} \le \frac{1}{2} ,
\]
or equivalently,
\[
  r \ge \frac{1}{4} \left(n \log^2(k/n^{1/3})\right)^{1/3} .
\]
Note that $k \le n^{1/3} e^n$ now implies that
$\eps \in (0, 1/\sqrt{2})$. With this choice, \lemref{assouad} implies
that
\begin{equation}
  \sR_n(\sF_k) \ge \frac{\eps}{4} \left(1 - \sqrt{\frac{n \eps^2}{r}} \right) \ge \frac{1}{32} \left( \frac{\log (k/n^{1/3})}{n} \right)^{1/3} . \eqlabel{monotone-lb-done}
\end{equation}
So we need only verify that these choices of $\eps$ and $r$ are
compatible with \eqref{monotone-cond}. Since $k \ge e^8 n^{1/3}$, then
there is an integer choice of $r$ in the range
\[
  \frac{1}{4} \left(n \log^2 (k/n^{1/3})\right)^{1/3} \le r \le \frac{1}{2} \left(n \log^2 (k/n^{1/3})\right)^{1/3} .
\]
In particular, we can verify that
\begin{align*}
  2r \le \left(n \log^2 (k/n^{1/3})\right)^{1/3} \le \frac{k}{2} ,
\end{align*}
since $2 \log^{2/3} x \le x$ for all $x \ge 0$. Moreover, since
$k \ge e^8 n^{1/3}$, then $\eps \ge \frac{1}{2 n^{1/3}}$, so that
\begin{align*}
  \frac{e^{4 \eps r}}{2 \eps} &\le n^{1/3} e^{4 \eps r} \\
                              &\le n^{1/3} e^{\frac{1}{2} \log (k/n^{1/3})} \\
                              &= n^{1/3} \sqrt{\frac{k}{n^{1/3}}} \\
                              &\le \frac{k}{2} ,
\end{align*}
where this last inequality holds since $\sqrt{x} \le x/2$ for all
$x \ge e^8$, so that \eqref{monotone-lb-done} is proved.

When $k \ge n^{1/3} e^n$, we argue by inclusion that
\[
  \sR_n(\sF_k) \ge \inf_{k \ge n^{1/3} e^n} \sR_n(\sF_k) \ge \frac{1}{32} .
\]

The only remaining case is $k \le e^8 n^{1/3}$. In this case, we offer
a different construction. Now, each $A_i$ will have size $2$ for
$1 \le i \le r$, where $r = \floor{k/2}$. Fix $a, b \in \R$ to be
specified later, and set
\begin{align*}
  g_{i 0}(x) &= \left\{\begin{array}{ll}
                        a - b(2i - 1) & \mbox{if $x \in A_i'$,} \\
                        a - b(2i + 1) & \mbox{if $x \in A_i''$,}
                      \end{array} \right. \\
  g_{i 1}(x) &= a - 2 b i ,
\end{align*}
We insist that
\[
  a - b (2r + 1) = \frac{1 - \eps}{2r}
\]
for some $0 \le \eps \le 1$. Since each $f_\theta$ must be a density,
we need that
\[
  \sum_{i = 1}^r 2(a - 2 bi) = 1 .
\]
Both of these conditions will be satisfied if we pick
\[
  b = \frac{\eps}{2 r^2}, \quad \text{and} \quad a = b + \frac{1 + \eps}{2r} ,
\]
Furthermore, the largest probability of an atom here is
\[
  a - b = \frac{1 + \eps}{2r} \le 1 ,
\]
for $k \ge 2$. Then, for $1 \le i \le r$, we can compute
\begin{align*}
  \int \left(\sqrt{f_\theta} - \sqrt{f_{\zeta_i}}\right)^2 &\le \frac{2 b^2}{a - 2 bi} \\
                                                           &\le \frac{\eps^2}{r^3 (1 - \eps)} ,
\end{align*}
so
\[
  \int \sqrt{f_\theta f_{\zeta_i}} \ge 1 - \frac{\eps^2}{2 r^3(1 - \eps)} .
\]
and
\[
  \int_{A_i} |f_\theta - f_{\zeta_i}| = 2b = \frac{\eps}{r^2} .
\]
Pick $\eps = e^{-12} r \sqrt{k/n}$. Then, since
$2 \le k \le e^8 n^{1/3}$ and $r = \floor{k/2} \ge k/3$, then
$\eps \le 1/2$, and
\[
  \sqrt{\frac{n \eps^2}{r^3 (1 - \eps)}} \le \frac{1}{2} ,
\]
so that
\[
  \pushQED{\qed}
  \sR_n(\sF_k) \ge \frac{\eps}{4 r} \left(1 - \sqrt{\frac{n \eps^2}{r^3 (1 - \eps)}}\right) \ge \frac{1}{8 e^{12}} \sqrt{k/n} . \qedhere
  \popQED
\]

\subsection{Proof of the lower bound in \thmref{convex-result}.}\applabel{convex-lower}

Let $A_1, \dots, A_r$ be the partition in \lemref{assouad}, for an
integer $r \ge 1$ to be specified. Let $j_i$ be the smallest element
of $A_i$, and suppose that each $|A_i|$ is chosen to be a positive
multiple of $3$. We will define the functions $f_\theta$ based on
parameters $\beta_i, \Delta_i \in \R$, to be specified. Let $g_{i0}$
linearly interpolate between the points
\[
  (j_i, \beta_i), \quad \text{and} \quad \left(j_i + \frac{|A_i|}{3}, \frac{\beta_{i + 1} - \beta_i}{3} + \beta_i - \Delta_i \right) ,
\]
on $\{j_i, j_i + 1, \dots, j_i + |A_i|/3 - 1\}$,
and between the points
\[
  \left( j_i + \frac{|A_i|}{3}, \frac{\beta_{i + 1} - \beta_i}{3} + \beta_i - \Delta_i \right), \quad \text{and} \quad (j_i + |A_i|, \beta_{i + 1}) 
\]
on $\{j_i + |A_i|/3, \dots, j_i + |A_i| - 1\}$. Let $g_{i1}$ linearly
interpolate between the points
\[
  (j_i, \beta_i), \quad \text{and} \quad \left(j_i + \frac{2 |A_i|}{3}, \frac{2 (\beta_{i + 1} - \beta_i)}{3} + \beta_i - \Delta_i \right)
\]
on $\{j_i, j_i + 1, \dots, j_i + 2|A_i|/3 - 1\}$, and
between the points
\[
  \left(j_i + \frac{2|A_i|}{3}, \frac{2 (\beta_{i + 1} - \beta_i)}{3} + \beta_i - \Delta_i \right), \quad \text{and} \quad (j_i + |A_i|, \beta_{i + 1})
\]
on $\{j_i + 2|A_i|/3, \dots, j_i + |A_i| - 1\}$.  Then, each
$f_\theta$ will be nonincreasing as long as
$\beta_i \ge \beta_{i + 1}$ for each $1 \le i \le r$, and
\begin{equation}
  \frac{2 (\beta_{i + 1} - \beta_i)}{3} + \beta_i - \Delta_i \ge \beta_{i + 1} \iff \frac{\beta_i - \beta_{i + 1}}{3} \ge \Delta_i . \eqlabel{mono-cond}
\end{equation}
Each $f_\theta$ will be convex as long as the largest slope on $A_i$
is at most the smallest slope on $A_{i + 1}$ for each $1 \le i \le
r$. Equivalently,
\begin{equation}
  \frac{|A_{i + 1}|}{|A_i|} \ge \frac{ \frac{\beta_{i + 1} - \beta_{i + 2}}{3} + \Delta_{i + 1}}{\frac{\beta_i - \beta_{i + 1}}{3} - \Delta_i} . \eqlabel{conv-cond}
\end{equation}
Now, pick $\beta_i = (1 - \eps)^{i - 1} \beta$ for some
$\beta \in \R$, $\eps \in (0, 1)$ to be specified, and
\[
  \Delta_i = \frac{(\beta_i - \beta_{i + 1}) \eps }{3} = \frac{ \beta \eps^2 (1 - \eps)^{i - 1}}{6} ,
\]
for which \eqref{mono-cond} is immediately satisfied. The condition
\eqref{conv-cond} is then equivalent to
\[
  \frac{|A_{i + 1}|}{|A_i|} \ge 1 + \eps .
\]
Pick $|A_1| = 3$. It is sufficient to make the choice
\[
  |A_i| \ge a_i = \frac{3}{(1 - \eps)^{i - 1}} .
\]
Let $|A_i|$ be the smallest integer multiple of $3$ at least as large
as $a_i$, so that $a_i \le |A_i| \le a_i + 3$. If $\eps \le 1/2$, then
\[
  \sum_{i = 1}^r |A_i| \le 3r + \frac{3 e^{2 \eps r}}{2 \eps} .
\]
Since the support of our densities is $\{1, \dots, k\}$, this upper
bound must not exceed $k$, so we impose that
\begin{equation}
  3r \le \frac{k}{2} , \quad \text{and} \quad \frac{ 3 e^{2 \eps r}}{2 \eps} \le \frac{k}{2} . \eqlabel{assouad-conv-conds}
\end{equation}
We must tune $\beta$ in order for each $f_\theta$ to be a density. By
monotonicity, we must have
\[
  1 \le \sum_{i = 1}^r \beta_i |A_i| \le 6 \beta r ,
\]
and
\[
  1 \ge \sum_{i = 1}^r \beta_{i + 1} |A_i| \ge 3(1 - \eps) \beta r ,
\]
so there is a choice of $\beta$ where
\[
  \frac{1}{6 r} \le \beta \le \frac{2}{3 r} ,
\]
as long as $\eps \le 1/2$. Now, fix $1 \le i \le r$. Then,
\[
  \int_{A_i} |f_\theta - f_{\zeta_i}| \ge \frac{|A_i| \Delta_i}{12} \ge \frac{\eps^2 \beta}{12} \ge \frac{\eps^2}{72 r} ,
\]
and
\begin{align*}
  \int (\sqrt{f_\theta} - \sqrt{f_{\zeta_i}})^2 &= \int_{A_i} \frac{(f_\theta - f_{\zeta_i})^2}{(\sqrt{f_{\theta}} + \sqrt{f_{\zeta_i}})^2} \\
                                                      &\le \frac{1}{4 \beta_{i + 1}} \int_{A_i} (f_\theta - f_{\zeta_i})^2 \\
                                                      &\le \frac{|A_i| \Delta_i^2}{4 \beta_{i + 1}} \\
                                                      &\le \frac{\left(3 + \frac{3}{(1 - \eps)^{i - 1}} \right) \left( \frac{\beta \eps^2 (1 - \eps)^{i - 1}}{3} \right)^2}{4 \beta (1 - \eps)^i} \\
                                                &\le \frac{2 \eps^4}{9 r} ,
\end{align*}
as long as $\eps \le 1/2$, whence
\[
  \int \sqrt{f_\theta f_{\zeta_i}} \ge 1 - \frac{\eps^4}{9 r} .
\]
Now, pick
\[
  \eps = \frac{1}{2} \left(\frac{\log (k/n^{1/5})}{n} \right)^{1/5} ,
\]
and $r$ for which
\[
  \sqrt{\frac{2 n \eps^4}{9r}} \le \frac{1}{2} ,
\]
or equivalently,
\[
  r \ge \frac{1}{18} n^{1/5} \log^{4/5} (k/n^{1/5}) .
\]
Note that $k \le n^{1/5} e^n$ now implies that $\eps \le 1/2$. With
this choice, \lemref{assouad} has
\begin{equation}
  \sR_n(\sG_k) \ge \frac{\eps^2}{288} \left(1 - \sqrt{\frac{2n \eps^4}{9r}} \right) \ge \frac{1}{1152} \left( \frac{\log (k/n^{1/5})}{n} \right)^{2/5} , \eqlabel{final-conv-lower}
\end{equation}
so it remains to verify that our choices of $\eps$ and $r$ are
compatible with \eqref{assouad-conv-conds}. Since $k \ge e^{40} n^{1/5}$,
then there is an integer choice of $r$ in the range
\[
  \frac{1}{18} n^{1/5} \log^{4/5} (k/n^{1/5}) \le r \le \frac{1}{9} n^{1/5} \log^{4/5} (k/n^{1/5}) .
\]
In particular, we can verify that
\[
  3r \le \frac{1}{3} n^{1/5} \log^{4/5} (k/n^{1/5}) \le \frac{k}{2} ,
\]
since $(2/3) \log^{4/5} x \le x$ for all $x \ge 0$. Moreover, since
$k \ge e^{40} n^{1/5}$, then $\eps \ge 1/n^{1/5}$, so that
\begin{align*}
  \frac{3 e^{2 \eps r}}{2 \eps} &\le \frac{3 n^{1/5} e^{2 \eps r}}{2} \\
                                &\le \frac{3 n^{1/5} e^{\frac{1}{9} \log (k/n^{1/5})}}{2} \\
                                &= \frac{3 n^{1/5}}{2} \left(\frac{k}{n^{1/5}}\right)^{1/9} \\
                                &\le \frac{k}{2} ,
\end{align*}
since $x^{1/9} \le x/3$ for all $x \ge e^{40}$, so that
\eqref{final-conv-lower} is proved.

When $k \ge n^{1/5} e^n$, we argue by inclusion that
\[
  \sR_n(\sG_k) \ge \inf_{k \ge n^{1/5} e^n} \sR_n(\sG_k) \ge \frac{1}{1152} .
\]

It remains to prove the case $k \le e^{40} n^{1/5}$. Observe that
$\sG_2 = \sF_2$, so the lower bound for $k = 2$ follows from
\appref{monotone-lower}, so we assume that $k \ge 3$. Now, each $A_i$
will have size $3$ for $1 \le i \le r$, where $r = \floor{k/3}$. Fix
$a, b \in \R$ to be specified later, and set
\begin{align*}
  g_{i0}(j_i) &= \beta_i  \\
  g_{i0}(j_i + 1) &= \frac{2 \beta_i + \beta_{i + 1}}{3} - \Delta_i \\
  g_{i0}(j_i + 2) &= \frac{\beta_i + 2 \beta_{i + 1}}{3} - \frac{\Delta_i}{2}
\end{align*}
and
\begin{align*}
  g_{i1}(j_i) &= \beta_i \\
  g_{i1}(j_i + 1) &= \frac{2 \beta_i + \beta_{i + 1}}{3} - \frac{\Delta_i}{2}  \\
  g_{i1}(j_i + 2) &= \frac{\beta_i + 2 \beta_{i + 1}}{3} - \Delta_i .
\end{align*}
Each $f_\theta$ will be non-increasing as long as
$\beta_i \ge \beta_{i + 1}$, and
\[
  \Delta_i \le \frac{\beta_i - \beta_{i + 1}}{3} ,
\]
for each $1 \le i \le r$. Convexity will follow if
\[
  \beta_{i + 1} - \left( \frac{\beta_i + 2 \beta_{i + 1}}{3} - \Delta_i \right) \le \left( \frac{2 \beta_{i + 1} + \beta_{i + 2}}{3} - \Delta_{i + 1}\right) - \beta_{i + 1} ,
\]
or equivalently,
\[
  \frac{\beta_{i} - \beta_{i + 1}}{3} - \Delta_i \ge \frac{\beta_{i + 1} - \beta_{i + 2}}{3} + \Delta_{i + 1} .
\]
We need also that $\beta_1 \le 1$, $\beta_{r + 1} \ge 0$, and
\[
  \sum_{i = 1}^r \left(2 \beta_i + \beta_{i + 1} - \frac{3 \Delta_i}{2}\right) = 1 .
\]
Take $\beta_{i + 1} = \beta_i - 3 \Delta_i - \alpha (r - i)$ for
$\alpha \ge 0$ to be specified.. Monotonicity follows, and convexity
will follow if
\begin{align*}
  \frac{\alpha (r - i)}{3}  &\ge 2 \Delta_{i + 1} + \frac{\alpha (r - i - 1)}{3} \\
  \iff  \Delta_{i + 1} &\le \frac{\alpha}{6} .
\end{align*}
So take each $\Delta_i = \alpha/6$. Then,
\[
  \beta_i = \beta_1 - \frac{\alpha (i - 1)}{2} -  \alpha \sum_{j = 1}^{i - 1} (r - i ) ,
\]
and in particular,
\[
  \beta_{r + 1} = \beta_1 - \frac{\alpha r}{2} - \frac{\alpha (r - 1) r}{2} = \beta_1 - \frac{\alpha r^2}{2} .
\]
Take $\alpha = \eps/r^3$ for some $0 \le \eps \le 1$, whence
\[
  \beta_{r + 1} = \beta_1 - \frac{\eps}{2r} .
\]
By monotonicity,
\[
  1 \le \sum_{i = 1}^r 3\beta_i \le 3 r \beta_1 ,
\]
and
\[
  1 \ge \sum_{i = 1}^r 3 \beta_{i + 1} \ge 3 r \left( \beta_1 - \frac{\eps}{2r} \right) \ge 3r \beta_1 - \frac{3 \eps }{2} ,
\]
so that the right choice of $\beta_1$ satisfies
\[
  \frac{1}{3r} \le \beta_1 \le \frac{5}{6r} .
\]
Fix $1 \le i \le r$. Then,
\begin{align*}
  \int_{A_i} |f_\theta - f_{\zeta_i}| = \Delta_i = \frac{\eps}{6 r^3} ,
\end{align*}
and if $\eps \le 1/2$,
\begin{align*}
  \int_{A_i} \left(\sqrt{f_\theta} - \sqrt{f_{\zeta_i}}\right)^2 \le \frac{\Delta_i^2}{8 \beta_{i + 1}} \le \frac{\eps^2}{24 r^5} .
\end{align*}
Finally, pick $\eps = e^{-100} r^2 \sqrt{k/n}$. Since
$k \le e^{40} n^{1/5}$ and $r = \floor{k/3} \ge k/6$, then
$\eps \le 1/2$, and
\[
  \sqrt{\frac{n \eps^2}{24 r^5}} \le \frac{1}{2} ,
\]
so by \lemref{assouad},
\[
  \pushQED{\qed}
  \sR_n(\sG_k) \ge \frac{\eps}{24 r^2} \left(1 - \sqrt{\frac{n \eps^2}{24 r^5}}\right) \ge \frac{1}{48 e^{100}} \sqrt{k/n} . \qedhere
  \popQED
\]